\documentclass[10pt]{amsart}
\usepackage{texmac}
\advance\textheight by 1mm
\operators{Sylow,Stab,Irr,Orb,Si,PSL,schur,FS,Qu,diag,cyc,Char,Sp,Perm,GCD}
\calsymbols{c}{Q,F,R,G}
\bbsymbols{b}{C}
\def\C{{\rm C}}
\def\CH{{\rm\hat{C}}}
\fraksymbols{f}{p}
\def\schur{m}

\long\def\comment#1\endcomment{}

\def\mat#1{\begingroup\smaller[2]
\baselineskip 8pt
\def\arraystretch{0.8}
\def\arraycolsep{4pt}
\left(\!
\begin{array}{cccccccccc}
#1
\end{array}
\right)\endgroup}

\begin{document}
\title[Rational representations and permutation representations]{Rational representations and permutation representations of finite groups}
\author{Alex Bartel}
\address{Mathematics Institute, University of Warwick,
Coventry CV4 7AL, UK}
\email{a.bartel@warwick.ac.uk}
\author{Tim Dokchitser}
\address{Department of Mathematics, University Walk, Bristol BS8 1TW, UK}
\email{tim.dokchitser@bristol.ac.uk}
\llap{.\hskip 10cm} \vskip -0.8cm

\maketitle

\begin{abstract}
We investigate the question which $\Q$-valued characters
and characters of $\Q$-representations of finite groups are
$\Z$-linear combinations of permutation characters. 
This question is known to reduce to that for quasi-elementary groups,
and we give a solution in that case. As one of the applications, we
exhibit a family of simple groups with rational representations
whose smallest multiple that is a permutation representation can be arbitrarily
large.
\end{abstract}

\section{Introduction}
\label{sIntro}

Many rational invariants of a finite group $G$ are encoded in the ring $\Char_\Q(G)$ 
of rationally-valued characters, the ring $R_\Q(G)$ of rational representations,
and the ring $\Perm(G)$ of virtual permutation representations. All three have the same 
$\Z$-rank, and there are natural inclusions with finite cokernels
$$
  \Perm(G) \quad\lar\quad R_\Q(G) \quad\lar\quad \Char_\Q(G).
$$
The quotient $\Char_\Q(G)/R_\Q(G)$ is studied by the theory of Schur indices,
and the purpose of this paper is to investigate the other two, 
$$
  \C(G) = \frac{R_\Q(G)}{\Perm(G)} \qquad\text{and}\qquad
  \CH(G) = \frac{\Char_\Q(G)}{\Perm(G)}.
$$
They have exponent dividing $|G|$ by Artin's induction theorem, 
and Serre remarked that $\C(G)$ need not be trivial (\cite{SerLi} Exc. 13.4). 
It is trivial for $p$-groups \cite{Feit,Ritter,Segal},
and it is known for nilpotent groups \cite{Ras} (see also \S\ref{sBasic}),
for Weyl groups of Lie groups
\cite{Sol,Kle} and in other special cases \cite{Berz,HT}. 
It follows from the general results of Dress, Kletzing, and Hambleton-Taylor-Williams \cite{Dress1, Dress2, Kle, HTW},
that the study of $\C(G)$ for a group $G$ reduces, in principle, 
to that of its quasi-elementary subgroups, or of its `basic' quasi-elementary subquotients. 
Specifically, for subgroups the statement is that of the two maps
$$
\prod_{Q\leq G\atop\text{quasi-elem.}}\C(Q)
  \quad\stackrel{\Ind}{\hbox to 2.5em{\rightarrowfill}}\quad 
 \C(G)
  \quad\stackrel{\Res}{\hbox to 2.5em{\rightarrowfill}}\quad 
\prod_{Q\leq G\atop\text{quasi-elem.}}\C(Q),
$$
the first one is surjective and the second one injective, and similarly for $\CH$.
This is also an immediate consequence of Solomon's induction theorem, see \S\ref{sMackey}.

Our first observation is that the composite map allows us to describe $\C(G)$ and $\CH(G)$
explicitly, in a way that bypasses the representation theory of $G$ --- purely
in terms of quasi-elementary subgroups and the `$\Res\Ind$' maps between them;
in fact,
it is enough to consider maximal quasi-elementary subgroups, i.e. $p$-normalisers of
cyclic subgroups of $ G$. In \S\ref{sMackey} we give a simple formula for the $\Res\Ind$ maps,
and in \S\ref{sQE} we prove one of the main results of the paper, 
which  describes $\C(Q)$ and 
$\CH(Q)$ for a $p$-quasi-elementary group $Q=C\rtimes P$. 

Its simplest formulation is:

\begin{theorem}[=Theorem \ref{thmqemain1}]
Let $\rho$ be an irreducible rational representation of 
a $p$-quasi-elementary group $Q=C\rtimes P$.
(So $C$ is cyclic, $P$ a $p$-group, and $p\nmid|C|$.)
The order of $\rho$ in $\C(Q)$ is $\frac{\dim\hat\psi\dim{\hat\pi}}{\dim\rho}$,
where $\hat\psi$ is the (unique) rational irreducible constituent of $\Res_C\rho$
and ${\hat\pi}$ a rational irreducible constituent of $\Res_P\rho$ of minimal dimension.
\end{theorem}

Together with the aforementioned `$\Res\Ind$' formula, it gives a way to compute 
$\C(G)$ and $\CH(G)$ efficiently in a given finite group $G$. 
Incidentally, it also gives an algorithm to find
$\Perm(G) \subset R_\Q(G)$ without computing the subgroup lattice,
which is now implemented in Magma \cite{Magma}. 
In \S\ref{sGL2} and \S\ref{sPSL} we illustrate
applications of this approach to proving both triviality and non-triviality of
$\C(G)$, as we shall now describe.

In general, $\C(G)$ remains somewhat mysterious, especially in non-soluble
groups. Already Frobenius knew that $\C(A_n)$ is trivial for
all $n$. It was announced by Solomon in \cite{Sol} that $\C(\PSL_2(\F_q))$
is trivial for all prime powers $q$. In \S\ref{sGL2} we explain how this,
and the same statement for $\GL_2(\F_q)$ and $\PGL_2(\F_q)$, follow from
the results of \S\ref{sMackey} and \S\ref{sQE}.

There is, to our knowledge, no example in the literature of a simple
group with non-trivial $\C(G)$. In \S\ref{sPSL} we show:

\begin{theorem}[=Theorem \ref{pslmain} and Corollary \ref{cor:psl}]
The exponent of the 2-part of $\C(G)$ is unbounded in the families $G=\PSL_k(\F_p)$
and $G=\SL_k(\F_p)$. Moreover,
$\CH(\PSL_{k}(\F_p))\ne\{1\}$ for all even $k\geq 4$ and all odd primes $p$.
\end{theorem}


%

\subsection*{Notation}

Throughout the paper, $G$ denotes a finite group. We write

\smallskip

\begin{tabular}{llllll}
  $\Char(G)$     & = & the character ring $G$, \cr
  $\Char_\Q(G)$  & = & the ring of $\Q$-valued characters, \cr
  $R_{\Q}(G)$    & = & the ring of characters of virtual $\Q G$-representations, \cr
  $\Perm(G)$     & = & the ring of characters of virtual permutation representations, \cr
  $\C(G)$        & = & ${R_\Q(G)}/{\Perm(G)}$, \cr
  $\CH(G)$       & = & ${\Char_\Q(G)}/{\Perm(G)}$, \cr
  $\Q(\chi)$     & = & the field of character values of a complex character $\chi$,\cr
  $\schur(\chi)$ & = & the Schur index of an irreducible complex character $\chi$ over $\Q(\chi)$.\\[2pt]
\end{tabular}

\noindent
For a complex character $\chi$ of $G$, define its \emph{trace} and,
when $\chi$ is irreducible, its \emph{rational hull} as
$$
\begin{array}{clllllllllll}
\tr\chi &=& \displaystyle \sum_{\scriptscriptstyle \sigma\in \Gal(\Q(\chi)/\Q)}\chi^{\sigma}&&\in \Char_\Q(G),  \\[5pt]
\hat{\chi} &=& \schur(\chi)\tr\chi &&\in R_\Q(G).\cr
\end{array}
$$
If $\chi$ is irreducible, then $\tr\chi$ is a \emph{$\Q$-irreducible character} and
$\hat{\chi}$ is the character of an \emph{irreducible rational representation}. We write

\smallskip

\begin{tabular}{llllll}
  $\Irr(G)$    & = & the set of (complex) irreducible characters of $G$, \cr
  $\Irr_\Q(G)$ & = & the set of $\Q$-irreducible characters of $G$, \cr
  $\mu(\alpha,\beta)$ & = & $\frac{\langle\alpha,\beta\rangle}{\langle\alpha,\alpha\rangle}$ = 
    multiplicity of $\alpha$ in $\beta$, \cr
  && used for characters $\alpha\in\Irr_\Q(G), \beta\in \Char_\Q(G)$, and\cr
  && also for rational representations $\alpha,\beta$ with $\alpha$ irreducible.\\[2pt]
\end{tabular}

\noindent
We write $x\sim y$ for conjugate elements. A $p$-quasi-elementary group is one of the form 
$G=C\rtimes P$ with $C$ cyclic, and $P$ a $p$-group; throughout the paper we adopt the 
convention that $p\nmid |C|$.

\begin{acknowledgements}
The first author is supported by a Research Fellowship from the Royal 
Commission for the Exhibition of 1851, and the second author is supported by a Royal Society
University Research Fellowship. We would like to thank Alexandre Turull 
for his help with Corollary \ref{cor:psl}. We are grateful to an
anonymous referee for a careful reading of the manuscript and many helpful
comments.
\end{acknowledgements}

%
%
%

\section{Basic facts}
\label{sBasic}

\begin{lemma}\label{lem:lifts}
An inclusion $N\normal G$ induces injections $\C(G/N)\injects \C(G)$,
$\CH(G/N)\injects \CH(G)$.
\end{lemma}
\begin{proof}
Suppose $\bar\rho$ is a representation of $G/N$, which lifts to $\rho\in\Perm G$. Write
$$
  \rho = \bigoplus \bC[G/H_i]^{\oplus n_i}, \qquad n_i\in\Z.
$$
For a subgroup $H\<G$ recall that
$
  \bC[G/H]^N \iso \bC[G/NH],
$
as $G$-representations (see e.g. \cite{tamroot}, proof of Thm. 2.8(5)). 
Therefore,
$$
  \bar\rho=\rho^N = \bigoplus \bC[G/NH_i]^{\oplus n_i}\in \Perm(G/N),
$$
as required.
\end{proof}

\begin{lemma}\label{lem:schurnoschur}
Let $\rho$ be an irreducible rational representation and $\tau\in\Irr G$ its 
constituent, so $\tr\tau\in \Irr_\Q(G)$ and $\rho=\schur(\tau)\tr\tau$. 
The order of $\tr\tau$ in $\CH(G)$ is $\schur(\tau)$ times the order of
$\rho$ in $\C(G)$.
\end{lemma}

\begin{proof}
Clear from the definitions of $\C(G)$ and $\CH(G)$.
\end{proof}

This allows us to immediately deduce results about $\CH(G)$ from those about $\C(G)$,
and conversely.

\subsection*{Nilpotent groups}

Some statements seem to have a cleaner formulation for $\C(G)$
than for $\CH(G)$, and for some it is the other way around. Let us briefly
illustrate this with an example of nilpotent groups:

\begin{theorem}[Rasmussen \cite{Ras} Thm 5.2]
Let $G=G_2\times G_{2'}$ be a nilpotent group, where $G_2$ is its Sylow 2-subgroup.
Then $\C(G)$ is trivial, unless $G_{2'}\neq \{1\}$ and there exists an
irreducible character $\chi$ of $G_2$ with $\schur(\chi)=2$ and such that
one of the following holds:
\begin{enumerate}
\item $\Q(\chi)\neq \Q$, or
\item $\Q(\chi)= \Q$ and there exists a prime divisor $q$ of $|G|$ such that the order
of 2 $\pmod q$ is even.
\end{enumerate}
\end{theorem}

The conditions turn out to be much simpler if one transforms this into a result
about $\CH(G)$. The following follows easily from \cite[Thm. 4.2]{Ras}
and standard facts about Schur indices:



\begin{theorem}
Let $\chi=\chi_2\chi_{2'}$ be an irreducible character of a nilpotent group
$G=G_2 \times G_{2'}$ as above. Then the order of $\tr\chi$ in $\CH(G)$ is 
$\schur(\chi_2)$ (which is 1 or~2).
\end{theorem}

\subsection*{Metabelian and supersolvable groups}

The following theorem will be of central importance in what follows. It implies that knowing 
the order of every $\Q$-irreducible representation in $\CH(G)$ determines the structure 
of $\CH(G)$ completely when $G$ is metabelian or supersoluble (e.g. nilpotent or 
quasi-elementary). It does not hold in arbitrary groups, as first noted by Berz \cite{Berz};
the smallest counterexample is $G=C_3\times\SL_2(\F_3)$.

\begin{theorem}[Berz \cite{Berz}]\label{thm:Berz}
If $G$ is metabelian or supersoluble, 
then $\Perm(G)\subseteq R_\Q(G)$ is freely generated by
$n_\rho\rho$, as $\rho$ ranges over irreducible rational representations of $G$, and
$$
  n_\rho =\gcd_{H\le G} \mu(\rho,\Q[G/H]).
$$
\end{theorem}


\begin{lemma}       
\label{A:Cp}
If $G=A\rtimes V$ with $A$ abelian and $V$ an elementary abelian
$p$-group, then $\CH(G)=\{1\}$.
\end{lemma}

\begin{proof}
By Theorem \ref{thm:Berz}, it is enough to show that every complex irreducible
character $\tau$ of $G$ occurs exactly once in $\bC[G/H]$ for a
suitable $H\<G$. This is clear when $\dim\tau=1$.
Otherwise $\tau=\Ind_{AU}^G\chi$, for some
subgroup $U$ of $V$ and a 1-dimensional character $\chi$ of $AU$
(see \cite[Part II, \S 8.2]{SerLi}). Let
$H$ be a subgroup of $V$ that is complementary to $U$, i.e. $HU=V$ and $H\cap U=\{1\}$.
By Mackey's formula, we have

\smallskip
$
  \hfill\langle \tau,\bC[G/H]\rangle = \langle \chi,\Res_{AU}\Ind_{H}^G\triv\rangle =
    \langle \chi,\Ind_{AU\cap H}^{AU}\triv\rangle = \langle \chi, \bC[AU]\rangle = 1.
\hfill
\qedhere
$
\end{proof}

Recall that a $p$-quasi-elementary group $G=C\rtimes P$ is {\em basic} if 
the kernel $K$ of $P\rightarrow \Aut(C)$ is trivial or isomorphic to $D_8$ or 
has normal $p$-rank one. 

\begin{proposition}[\cite{HT}, Proposition 5.2]\label{prop:mainresfaithful}
Let $G=C\rtimes P$ be basic $p$-quasi-elementary.
Let $A_p$ be a maximal cyclic
subgroup of $K=\ker(P\rightarrow \Aut(C))$ that is normal in $P$ (it is all of $K$ if $K$ is cyclic,
and has index 2 in $K$ otherwise), let $A=C A_p$, 
and let $\chi$ be a faithful one-dimensional character of $A$.
Then $\rho=\tr \Ind^G_A\chi$ is a $\Q$-irreducible character, and
$$
  \text{order of $\rho$ in $\CH(G)$} = \frac{|P|}{|A_p|\cdot 
    \max\limits_{H\leq P\atop H\cap A_p=1}|H|}. 
$$
\end{proposition}

\section{$\CH(G)$ as a Mackey functor}
\label{sMackey}

Let $\cR$ be a Mackey subfunctor of the character ring
Mackey functor $\Char(G)$. This simply means that for any finite group
$G$, $\cR(G)$ is a subgroup of $\Char(G)$ such that
if $H\leq G$ are finite groups, then
\begin{itemize}
\item for all $\rho\in \cR(H)$, $\Ind^G_H\rho\in \cR(G)$,
\item for all $\tau\in \cR(G)$, $\Res_H\tau\in \cR(H)$,
\item for all $\rho\in \cR(H)$ and $g\in G$, $\rho^g\in \cR(H^g)$.
\end{itemize}

Here are some examples:
\begin{itemize}
\item $R_K(G)$, the representation ring
of $G$ over a fixed subfield $K$ of $\bC$,
\item $\Char_K(G)$, the ring generated by $K$-valued characters, with fixed $K\subset \bC$,
\item $\Perm(G)$, the ring generated by permutation characters,
\item the subgroup of $\Char(G)$ generated by
characters of degree divisible by a fixed integer $n$.
\end{itemize}

If $p$ is a prime number, write $\cR(G)_p$ for $\cR(G)\otimes \Z_p$, where
$\Z_p$ is the ring of $p$-adic integers.

\begin{proposition}\label{mainreductionp}
Let $G$ be a finite group, fix a prime number $p$, and let $\cF_p$ be a family
of subgroups of $G$ such that every $p$-quasi-elementary subgroup of $G$ is conjugate to
a subgroup of some $Q\in\cF_p$. Then 
$$
  \prod_{Q\in \cF_p}\Res_Q: \cR(G)_p \lar \prod_{Q\in\cF_p} \cR(Q)_p
$$
is injective.
Dually,
$$
  \sum_Q\Ind_Q^G: \prod_{Q\in\cF_p} \cR(Q)_p \lar \cR(G)_p
$$
is surjective.
\end{proposition}
\begin{proof}
By Solomon's induction theorem,
a prime-to-$p$ multiple $d$ of the trivial representation can be written as
$$
  d\triv_G = \sum_i n_i \Ind_{H_i}^G\triv_{H_i}
$$
for some $p$-quasi-elementary subgroups $H_i$ and integers $n_i$. 
Because $\Ind_{H_i^g}^G\triv_{H_i^g}\iso \Ind_{H_i}^G\triv_{H_i}$, we may assume
that each $H_i$ is contained in some $Q_i\in\cF_p$. Taking tensor products with
any $\rho\in \cR(G)$ yields
$$
  d\rho = \sum_i n_i \Ind_{H_i}^G\Res_{H_i}\rho.
$$
If all $\Res_{H_i}\rho$ were 0, then so would be $d\rho$, and therefore also $\rho$.
This proves injectivity. Also, the equation shows that $d\rho \in
\Im\left(\sum_Q\Ind_Q^G\cR(Q)\right)$, which proves surjectivity, since $d$ is invertible
in $\Z_p$.
\end{proof}

\begin{corollary}
\label{mainreductionindres}
For $S,T\in\cF_p$ write $\alpha_{S,T}=\Res_T \Ind^G_S: \CH(S)\lar\CH(T)$.
Then
$$
  \CH(G)_p \iso \text{Image}\Bigl(\prod_T\sum_S \alpha_{S,T}: \prod_{S\in\cF_p} \CH(S) \lar \prod_{T\in\cF_p} \CH(T)\Bigr).
$$
In particular,
$\CH(G)_p=1$ if and only if
for all pairs $S,T\in \cF_p$ and all $\rho\in R_\Q(S)$ (equivalently, for those $\rho$ whose class in
$\C(S)$ is nontrivial), we have $\Res^G_T\Ind_S^G\rho \in \Perm(T)$. The same also
holds for $\C(G)$.
\end{corollary}
\begin{proof}
Apply Proposition \ref{mainreductionp} to $\cR$ being $\Perm,R_{\Q}$,
and $\Char_{\Q}$.
\end{proof}

\begin{corollary}
\label{mainreduction}
Let $\cF$ be a family of subgroups of $G$ such that every quasi-elementary
subgroup is conjugate to a subgroup of some $Q\in\cF$.
Then 
$$
  \CH(G) \longinjects \prod_{Q\in\cF} \CH(Q)
$$ 
via the (product of) restriction maps. Consequently, 
the kernel of the composition
$$
  R_\Q(G) \>{\buildrel\prod\Res\over{\hbox to 40pt{\rightarrowfill}}}\> 
    \prod_{Q\in\cF} R_\Q(Q) \lar \prod_{Q\in\cF} \CH(Q)
$$
is $\Perm(G)$. Dually, the composition 
$$
  \prod_{Q\in\cF} R_\Q(Q) \stackrel{\Ind}{\longrightarrow} R_\Q(G) \rightarrow \CH(G)
$$
is onto. The same holds with $R_\Q$ replaced by $\Char_\Q$ and $\hat C$ by $C$.
\end{corollary}



\begin{remark}
The theorem and the two corollaries give a very efficient way of 
computing $\CH(G)_p, \CH(G)$, $\C(G)_p$, $\C(G)$ and of finding 
$\Perm(G)$ as a subring of $R_\Q(G)\leq \Char_\Q(G)$,
without computing the full lattice of subgroups of $G$.
\end{remark}

\begin{remark}
One possible family $\cF_p$ is the set of maximal $p$-quasi-elementary 
subgroups of $G$. These are of the form
$$
  Q = C\rtimes \Syl_p(N_G(C)),
$$
where $C$ is cyclic of order prime to $p$. Possible families $\cF$ in Corollary
\ref{mainreduction} are $\cF=\bigcup_p \cF_p$,
as $p$ ranges over prime divisors of $|G|$,
or alternatively $\cF=\{N_G(C)\}$ as $C$ ranges over 
(representatives of conjugacy classes of) cyclic subgroups of $G$.
\end{remark}


\begin{notation}
For the remainder of this section we use the following notation:

\smallskip

\begin{tabular}{lll}
$CC(G)$         & = & the set of conjugacy classes of $G$,\\
$CC_{\cyc}(G)$  & = & the set of conjugacy classes of cyclic subgroups of $G$,\\
$[x]$           & = & the conjugacy class of $x$, when $x$ is either an element
of $G$\\ & & or a cyclic subgroup,\\
$\tr^*\chi$     & = & the normalised trace
$\tr^*\chi=\frac{1}{[\Q(\chi):\Q]}\tr\chi$ of a character $\chi$,\\
$\tau(D)$       & = & $\tau(y)$, where $D\leq G$ is a cyclic subgroup,
$y$ is any generator\\
& & of $D$, and $\tau\in \Char_{\Q}(G)\otimes \Q$. The rationality
of $\tau$ ensures\\
 & & that $\tau(y)$ only depends on $D$ and not on the generator $y$.
\end{tabular}

Note in particular, that
for any character $\chi$ of $G$ and any cyclic subgroup $D$ of $G$,
$\tr^*\chi(D)$ is the average value of $\chi$ on the generators of $D$.
\end{notation}

\begin{lemma}
\label{lem:innerprod}
Let $H_1$, $H_2$ be two subgroups
of $G$. Let $\tau_i$ be a character of $H_i$, $i=1,2$, and
assume that
$\tau_1$ is $\Q$-valued. Then
\beq
  \lefteqn{\langle \Ind_{H_1}^G \tau_1,  \Ind_{H_2}^G \tau_2\rangle=}\\
  & & \displaystyle\frac{1}{|H_1||H_2|} \sum_{[C]\in CC_{\cyc}(G)}|N_G(C)|\phi(|C|)\cdot
  \sum_{D_1\leq H_1\atop D_1\sim C}\tau_1(D_1)\cdot
  \sum_{D_2\leq H_2\atop D_2\sim C}\tr^*\tau_2(D_2).
\eeq
\end{lemma}
\begin{proof}
First, note that by definition of inner products and of induced class functions,
$$
  \langle \Ind_{H_1}^G \tau_1,  \Ind_{H_2}^G \tau_2\rangle
    =
  \frac{1}{|H_1||H_2|}
  \sum_{[x]\in CC(G)} |Z_G(x)| 
    \overline{\Bigl(\sum_{y\in[x]\cap H_1} \tau_1(y)\Bigr)}
    \Bigl(\sum_{y\in[x]\cap H_2} \tau_2(y)\Bigr).
$$
The idea of the proof is to partition the set of conjugacy classes of
elements of $G$ according to conjugacy classes of cyclic subgroups they generate,
and to use the fact that for a rational character $\tau$, $\tau(x) = \tau(x')$
whenever $x$ and $x'$ generate conjugate cyclic subgroups. We get
\beq
\lefteqn{\langle \Ind_{H_1}^G \tau_1,  \Ind_{H_2}^G \tau_2\rangle}\\
& = & \displaystyle \frac{1}{|H_1||H_2|}
  \sum_{[x]\in CC(G)} |Z_G(x)| 
    \overline{\Bigl(\sum_{y\in[x]\cap H_1} \tau_1(y)\Bigr)}
    \Bigl(\sum_{y\in[x]\cap H_2} \tau_2(y)\Bigr)\\[10pt]
& = &\displaystyle \frac{1}{|H_1||H_2|}
  \sum_{[C] \in CC_{\cyc}(G)}f(C),
\eeq
where
\beq
  f(C) &=& \displaystyle |Z_G(C)|\cdot \sum_{[x]\in CC(G)\atop \langle x\rangle = C}
  \Bigl(\sum_{y\in[x]\cap H_1} \tau_1(y)\Bigr)
    \Bigl(\sum_{y\in[x]\cap H_2} \tau_2(y)\Bigr)\\[10pt]
& = & |Z_G(C)|\cdot \#\{k: x\sim x^k\}\cdot\\
  & &\displaystyle\qquad \cdot \sum_{D_1\leq H_1\atop D_1\sim C}\tau_1(D_1)\cdot
  \sum_{[x]\in CC(G)\atop [x]\sim C}\sum_{y\in[x]\cap H_2} \tau_2(y)\\[10pt]
& = & \displaystyle|N_G(C)|\cdot
  \sum_{D_1\leq H_1\atop D_1\sim C}\tau_1(D_1)\cdot
  \sum_{D_2\leq H_2\atop D_2\sim C}\sum_{\text{generators}\atop y \text{ of }D_2}\tau_2(y)\\[10pt]
& = &\displaystyle |N_G(C)|\cdot
  \sum_{D_1\leq H_1\atop D_1\sim C}\tau_1(D_1)\cdot
  \sum_{D_2\leq H_2\atop D_2\sim C}\phi(|C|)\tr^*_{\bar{\Q}/\Q}\tau_2(y),
\eeq
as claimed.
\end{proof}
\begin{corollary}\label{cor:mu}
Suppose $H_1\<Q_1\<G$, $H_2\<Q_2\<G$, and let 
$\chi_i$ be irreducible characters of $H_i$. 
Set $\tau_i = \Ind_{H_i}^{Q_i}\chi_i$,
and $\rho_i=\tr\tau_i$.
Assume that $\tau_2$ is irreducible.~Then
\beq
\mu(\rho_2,\Res_{Q_2}\Ind^{G}_{Q_1}\rho_1) & = &
\displaystyle\frac{[\Q(\tau_1):\Q]}{|H_1|\cdot|H_2|}
  \sum_{[C]\in CC_{\cyc}(G)}|N_G(C)|\phi(|C|)\cdot\\
  & & \displaystyle\qquad\cdot\sum_{D_1\leq H_1\atop D_1\sim C}\tr^*\chi_1(D_1)\cdot
  \sum_{D_2\leq H_2\atop D_2\sim C}\tr^*\chi_2(D_2).
\eeq
\end{corollary}
\begin{proof}
\beq
\mu(\rho_2,\Res_{Q_2}\Ind^{G}_{Q_1}\rho_1)
  & = & \langle \tau_2, \Res_{Q_2}\Ind^G(\sum \tau_1^\sigma)\rangle\\
  &=& \langle \Ind^G_{H_2}\chi_2,
  \Ind_{H_1}^G (\sum_{\sigma\in \Gal(\Q(\tau_1)/\Q)}
  (\chi_1)^\sigma)\rangle\\
  &=& \frac{1}{[\Q(\chi_1):\Q(\tau_1)]}\langle \Ind^G_{H_2}\chi_2,
  \Ind_{H_1}^G
  (\tr\chi_1)\rangle\\[10pt]
  &=&\frac{1}{[\Q(\chi_1):\Q(\tau_1)]|H_1|\cdot|H_2|}
  \sum_{[C]\in CC_{\cyc}(G)}|N_G(C)|\phi(|C|)\cdot \\
  & & \qquad \cdot \sum_{D_1\leq H_1\atop D_1\sim C}\tr\chi_1(D_1)\cdot
   \sum_{D_2\leq H_2\atop D_2\sim C}\tr^*\chi_2(y)\\[10pt]
  &=&\frac{[\Q(\tau_1):\Q]}{|H_1|\cdot|H_2|}
  \sum_{[C]\in CC_{\cyc}(G)}|N_G(C)|\phi(|C|)\cdot\\
  & & \qquad\cdot 
  \sum_{D_1\leq H_1\atop D_1\sim C}\tr^*\chi_1(D_1)\cdot
  \sum_{D_2\leq H_2\atop D_2\sim C}\tr^*\chi_2(D_2).
\eeq
\end{proof}

\begin{lemma}\label{lem:trace}
If $C$ is a cyclic group, and $\chi$ is a 1-dimensional character of $C$, 
then $(\tr^*\chi)(C) = \mu(\ord(\chi))/\phi(\ord(\chi))$,
where $\mu$ is the Moebius mu function, and $\ord(\chi)$ is the smallest natural number
$n$ such that $\chi^{n}=\triv$.
\end{lemma}
\begin{proof}
It is enough to prove the lemma for faithful characters $\chi$, since we
may, without loss of generality, replace $C$ by $C/\ker\chi$.
Let $g$ be a generator of $C$. Then
$$
(\tr^*\chi)(C) = \frac{1}{[\Q(\chi):\Q])}\tr\chi(g)=\frac{1}{\phi(\ord(\chi))}\tr\chi(g).
$$
If $|C|=n$, then $\chi(g)$ is a primitive $n$-th root of unity, and the fact that
its trace is $\mu(n)$ is classical.
\end{proof}

\begin{corollary}
\label{cor:muindres}
Let $G$ be a group and $p^r$ a prime power.
Then $\CH(G)$ has an element of order $p^r$
if and only if there exist two $p$-quasi-elementary subgroups $Q_1$,
$Q_2$ of $G$, irreducible monomial characters $\tau_i=\Ind_{H_i}^{Q_i}\chi_i$
of $Q_i$, and an integer $k$,
such that
\begin{itemize}
\item
the rational character $\tr\tau_2$ has order divisible by
$p^{k+r}$ in $\CH(Q_2)$, and
\item 
the rational number
\beq
\lefteqn{\frac{[\Q(\tau_1):\Q]}{|H_1||H_2|}\cdot\!\!\!\!
  \sum_{[U]\in CC_{\cyc}(G)}|N_G(U)|\phi(|U|)\cdot}\\
  & &\displaystyle\cdot \sum_{D_1\leq H_1\atop D_1\sim U}
  \frac{\mu([D_1:D_1\cap\ker\chi_1])}{\phi([D_1:D_1\cap\ker\chi_1])}\cdot
  \sum_{D_2\leq H_2\atop D_2\sim U}
    \frac{\mu([D_2:D_2\cap\ker\chi_2])}{\phi([D_2:D_2\cap\ker\chi_2])}
\eeq
has $p$-adic valuation at most $k$.
\end{itemize}
In this case, $\Ind^G_{Q_1}\tr{\tau}_1$ has order divisible by $p^r$ in $\CH(G)$.
\end{corollary}

\begin{remark}\label{simplification}\par\noindent
\begin{itemize}
\item 
Note that it is enough to take the last two sums in the above
formula only over those $D_i$ for which $D_i\cap \ker \chi_i$ has square-free
index in $D_i$, since for the others $\mu(\ord(\Res_{D_i}\chi_i))=0$.
For example if $\chi_i$ are faithful, then the outer sum may be taken over $U$ of
square free order.
\item
If, say, $H_1$ is cyclic, the sum $\sum_{D_1\leq H_1\atop D_1\sim U}$ has 
at most one term for every $U$.
\item
If $Q_1$, $Q_2$ are basic and $H_1$, $H_2$ are cyclic, then
Proposition~\ref{prop:mainresfaithful} gives a simple expression for the 
order of $\tr\tau_2$ in $\CH(Q_2)$.
\end{itemize}
\end{remark}

\begin{proof}[Proof of Corollary \ref{cor:muindres}]
By Corollary \ref{mainreductionindres}, $\CH(G)_p$ has an element of order
$p^r$ if and only if there exist $p$-quasi-elementary subgroups $Q_1$,
$Q_2$, and characters $\rho_i\in \Irr_{\Q}(Q_i)$,
such that $\rho_2$ has order $p^{k+r}$ in $\CH(Q_2)$ for some 
$k$, and $\mu(\rho_2, \Res_{Q_2}\Ind^G\rho_1)$ has $p$-adic valuation at most
$k$. Quasi-elementary groups are M-groups, so if $\tau_i$ is a complex
irreducible constituent of $\rho_i$, then there exist subgroups
$H_i\leq Q_i$ such that $\tau_i=\Ind_{H_i}^{Q_i}\chi_i$
for 1-dimensional characters $\chi_i\in\Irr(H_i)$.
The result therefore follows from Corollary \ref{cor:mu} in combination
with Lemma \ref{lem:trace}.
\end{proof}

\section{Quasi-elementary groups}\label{sQE}

The aim of this section is to provide several formulae of theoretical and
algorithmic interest for the orders of characters in 
$\CH(G)$ and $\C(G)$ when $G$ is quasi-elementary. Let
$G=C\rtimes P$ with $P$ a $p$-group
and $C$ cyclic of order coprime to $p$; we identify $P$ with a Sylow
subgroup of $G$.


\begin{lemma}\label{etatheta}
  Let $N$ be a normal subgroup of a finite group $G$, let $\eta$ be an
  irreducible character of $N$, and let $\theta$ be a complex irreducible
  constituent of $\Ind_N^G\eta$. Write $\cG_\eta=\Gal(\Q(\eta)/\Q)$, and similarly
  for $\cG_\theta$. Then
  $$
    \frac{[\Q(\eta):\Q]}{[\Q(\theta):\Q]} =
    \frac{\#\{\gamma\in \cG_\eta\;|\;\langle\eta^\gamma,\Res_N\theta\rangle\neq 0\}}
    {\#\{\gamma\in \cG_\theta\;|\;\langle\Ind_N^G\eta,\theta^\gamma\rangle\neq 0\}}.
  $$
  In particular, if $\Ind_N^G\eta$ is irreducible, then 
  $$
  \frac{[\Q(\eta):\Q]}{[\Q(\theta):\Q]} = \#\{\gamma\in \cG_\eta\;|\;\langle\eta^\gamma,\Res_N\theta\rangle\neq 0\}.
  $$
\end{lemma}
\begin{proof}
  The $G$-action on the characters of $N$ commutes with the Galois action.
  Every Galois conjugate of $\theta$ is a constituent of $\Ind^G\eta^\gamma$
  for some $\gamma\in \cG_\eta$, and moreover the 
  number of distinct Galois conjugates of $\theta$ in $\eta^\gamma$ is independent
  of $\gamma$. Also,
  the number of Galois conjugates of $\eta$ in $\Res_N\theta^\gamma$ is
  independent of $\gamma\in\cG_{\theta}$. So an inclusion--exclusion count gives
  $$
  \#\cG_\theta=\#\cG_{\eta}\cdot
  \frac{\#\{\gamma\in \cG_\theta\;|\;\langle\Ind_N^G\eta,\theta^\gamma\rangle\neq 0\}}
  {\#\{\gamma\in \cG_\eta\;|\;\langle\eta^\gamma,\Res_N\theta\rangle\neq 0\}}.
  $$
\end{proof}

\begin{lemma}\label{dimhull}
Let $\eta$ be an irreducible complex representation of $G$, with rational hull $\hat\eta$.
Then
$$
  \dim\hat\eta = \dim \eta\cdot\schur(\eta)\cdot[\Q(\eta):\Q].
$$
\end{lemma}
\begin{proof}
  The rational hull of $\eta$ is given by
  $$
  \hat\eta = \schur(\eta)\sum_{\gamma\in\Gal(\Q(\eta)/\Q)}\eta^\gamma,
  $$
  whence the claim follows.
\end{proof}

\begin{theorem}\label{thm:innerprod}
Let $G=C\rtimes X$ with $C$ cyclic of order coprime to $|X|$. Let
$\tau$ be a complex irreducible character of $G$ with rational hull
$\rho=\hat\tau$,
let $\pi$ be a complex
irreducible constituent of $\Res_X\tau$ with rational hull $\hat{\pi}$,
$\psi$ an irreducible constituent of $\Res_C\tau$ with rational hull
$\hat\psi$, $K_{\psi}$ the stabiliser of $\psi$
under the $X$-action on $\Irr(C)$, and let $\xi$ be a complex irreducible
constituent of $\Res_{K_{\psi}}\pi$.
Then
\begin{eqnarray*}
  \lefteqn{\mu({\rho},\Ind_X^G{\hat\pi}) =}\\
  & & = \frac{\schur(\pi)}{\schur(\tau)}
     \langle\xi,\Res_{K_{{\psi}}}\pi\rangle\cdot 
     \#\{\text{Galois conjugates $\pi'$ of $\pi \;|\; \langle\Res_{K_{\psi}}\pi',\xi\rangle\neq 0$}\}\\
     & & = \frac{\dim{{\hat\psi}}\dim{\hat\pi}}{\dim{\rho}}.
\end{eqnarray*}

\end{theorem}
\begin{proof}
%
We may assume that $\rho|_C$ is faithful, otherwise we prove the
result in the quotient $G/(\ker\rho\cap C)$. So
$K=K_{\psi}$ is assumed to be the kernel of the $X$-action on $C$.
Recall that $\psi$ denotes a complex constituent of $\tau|_C$.
In particular, $\tau=\Ind_{CK}^G \psi\xi$, as explained in \cite[Part II, \S 8.2]{SerLi}.
We have
\beq
  \rho &=& \displaystyle\schur(\tau)
  \sum\limits_{\scriptscriptstyle\gamma\in\Gal(\Q(\tau)/\Q)} \tau^\gamma; &&
    \dim\rho=\schur(\tau)[\Q(\tau):\Q]\dim\tau,\cr
  \hat\pi &=& \schur(\pi)\displaystyle
  \sum\limits_{\scriptscriptstyle\gamma\in\Gal(\Q(\pi)/\Q)} \pi^\gamma; &&
    \dim\hat\pi=\schur(\pi)[\Q(\pi):\Q]\dim\pi.\cr
\eeq
Thus
\beq
  \mu(\rho,\Ind_X^G\hat\pi) & = & \frac{1}{\schur(\tau)} \langle \tau, \Ind_X^G \hat\pi\rangle
  =
  \frac{1}{\schur(\tau)} \langle \Ind_{CK}^G \psi\xi, \Ind_X^G \hat\pi\rangle\\[4pt]
  &=& \frac{1}{\schur(\tau)} \langle \Res_X \Ind_{CK}^G \psi\xi, \hat\pi\rangle
  = \frac{1}{\schur(\tau)} \langle \Ind_K^X \xi, \hat\pi\rangle
  = \frac{1}{\schur(\tau)} \langle \xi, \Res_K \hat\pi\rangle,\cr
\eeq
where the last line follows from Mackey's formula, noting that
$CK\backslash G/X$ consists of one double coset, and that
$CK\cap X=K$.

Next, $X$ acts on the representations of $K$ by conjugation, 
and there is a Clifford theory decomposition
\beql{piclif}
  \Res_K \pi = e\sum\limits_{\scriptscriptstyle g\in X/\Stab_X\xi} \xi^g.
\eeql
Recall that the constituents of $\hat\pi$ are Galois conjugates
of $\pi$, and we select those whose restriction to $K$ contains $\xi$:
$$
  \Omega = \bigl\{\gamma\in\Gal(\Q(\pi):\Q) \bigm| 
     \langle \Res_K\pi^\gamma,\xi\rangle\ne 0\bigr\}.
$$
The inner product $\langle \Res_K\pi^\gamma,\xi\rangle=\langle \Res_K\pi,\xi^{\gamma^{-1}}\rangle$
is the same (and equals $e$) for every $\gamma\in\Omega$, since 
$\xi^{\gamma^{-1}}$ is irreducible and so must be one of $\xi^g$ in \eqref{piclif}.
So we have
$$
  \frac{1}{\schur(\tau)} \langle \xi, \Res_K \hat\pi\rangle 
    = \frac{\schur(\pi)}{\schur(\tau)} |\Omega| \langle \xi, \Res_K \pi\rangle,
$$
which proves the first equality.

It remains to show that
\beql{mudims}
  \frac{\schur(\pi)}{\schur(\tau)} |\Omega| \langle \xi, \Res_K \pi\rangle = \frac{\dim{\hat\psi}\dim\hat\pi}{\dim\rho}.
\eeql

By comparing dimensions in \eqref{piclif}, and since $\tau=\Ind_{CK}^G\psi\xi$,
we see that
$$
  \langle \xi, \Res_K \pi\rangle = e = \frac{\dim\pi}{[X:\Stab_X\xi]\dim \xi}
  = \frac{[X:K]\dim\pi}{[X:\Stab_X\xi]\dim \tau} = \frac{[\Stab_X\xi:K]\dim\pi}{\dim\tau},
$$
so
$$
  \mu(\rho,\Ind^G\hat\pi) = \frac{\schur(\pi)}{\schur(\tau)} |\Omega| \langle \xi, \Res_K \pi\rangle =
  |\Omega|\cdot [\Stab_X\xi:K]\frac{\schur(\pi)\dim\pi}{\schur(\tau)\dim\tau}.
$$

Consider the two groups
\begin{eqnarray*}
H_1 & = & \{\gamma\in \Gal(\Q(\psi\xi)/\Q) \; |\; \langle (\psi\xi)^\gamma, \Res_{CK}\tau\rangle\neq 0\},\\
H_2 & = & \{\gamma\in \Gal(\Q(\xi)/\Q)\;|\;\langle\xi^\gamma,\Res_K\pi\rangle\neq 0\}.
\end{eqnarray*}
There is a natural projection $H_1\surjects H_2$ given by the restriction of
Galois action to $\Q(\xi)$, whose kernel consists of precisely those elements
of $\Gal(\Q(\psi\xi)/\Q)$ that act trivially on $\xi$, and through the action
of some $g\in X$ on $\psi$ (this last condition is equivalent to the Galois
element being in $H_1$). Thus, the kernel is isomorphic to the subgroup
of $G/CK$ that acts trivially on $\xi$, i.e. to $\Stab_X\xi/K$. We
deduce that 
$$
\mu(\rho,\Ind^G\hat\pi) = |\Omega| \frac{|H_1|}{|H_2|}\frac{\schur(\pi)\dim\pi}{\schur(\tau)\dim\tau}.
$$

Now, by applying Lemma \ref{etatheta} first to $CK\normal G$ with $\theta=\tau$,
$\eta=\psi\xi$,
and then to $K\normal X$ with $\theta = \pi$, $\eta=\xi$, we find that
$$
|H_1|  =  \frac{[\Q(\xi):\Q][\Q(\psi):Q]}{[\Q(\tau):\Q]} \qquad\text{and}\qquad
|H_2|  =  |\Omega|\frac{[\Q(\xi):\Q]}{[\Q(\pi):\Q]},
$$
so that
\begin{eqnarray*}
\mu(\rho,\Ind^G\hat\pi) & = & |\Omega| \cdot\frac{|H_1|}{|H_2|}\frac{\schur(\pi)\dim\pi}{\schur(\tau)\dim\tau}\\
& = &|\Omega|\cdot \frac{[\Q(\xi):\Q]\cdot[\Q(\psi):\Q]\big/[\Q(\tau):\Q]}
{|\Omega|[\Q(\xi):\Q]\big/[\Q(\pi):\Q]}
\cdot\frac{\schur(\pi)\dim\pi}{\schur(\tau)\dim\tau}\\
 & = & \frac{[\Q(\psi):\Q]\cdot[\Q(\pi):\Q]\schur(\pi)\dim\pi}{[\Q(\tau):\Q]\schur(\tau)\dim\tau}
  =  \frac{\dim{\hat\psi}\dim\hat\pi}{\dim\rho},
\end{eqnarray*}
where the last equality follows from Lemma \ref{dimhull}.
\end{proof}

\begin{theorem}
\label{thmqemain1}
Let $G=C\rtimes P$ be $p$-quasi-elementary,
let $\rho$ be an irreducible rational representation of $G$. Let $\psi$ be a
complex irreducible constituent of $\Res_C\rho$ with rational hull $\hat\psi$,
and let $\hat\pi$ be a rational irreducible constituent of
$\Res_P\rho$ of minimal dimension.
Denote by $\pi$ a complex irreducible constituent of $\hat\pi$,
by $\xi$ a complex irreducible consitutent of $\pi|_{K_\psi}$, where
$K_{\psi}\leq P$ is the stabiliser in $P$ of $\psi$,
and by $\tau$ a complex irreducible constituent of $\rho$ such that $\Res_P\tau$ contains $\pi$. Then
\begin{eqnarray*}
  \lefteqn{\text{order of $\rho$ in $\C(G)$}
     =\mu(\rho,\Ind_P^G{\hat\pi})= \frac{\dim\hat\psi\dim{\hat\pi}}{\dim\rho}}\\
    & & =\frac{\schur(\pi)}{\schur(\tau)}
     \langle\xi,\Res_{K_\psi}\pi\rangle\cdot 
     \#\{\text{Galois conjugates $\pi'$ of $\pi \; |\; \xi\subset\Res_{K_\psi}\pi'$}\}.
\end{eqnarray*}
\end{theorem}

\begin{proof}
We may assume that $\rho|_C$ is faithful, otherwise we prove the
result in the quotient $G/(\ker\rho\cap C)$ (see Lemma \ref{lem:lifts}). Thus,
$K=K_\psi$ is assumed to be the kernel of the $P$-action on $C$.
Under this assumption, if $H\le G$ intersects $C$ non-trivially, then
$$
  \langle\rho,\bC[G/H]\rangle_G = \langle \Res_H\rho, \triv\rangle_H=0.
$$
Write $o$ for the order of $\rho$ in $\C(G)$. 
By Theorem \ref{thm:Berz}, we have
\beq
  o\cdot\langle \rho,\rho\rangle &=&
    \gcd\limits_{H\le G}\langle\rho,\bC[G/H]\rangle_G =
    \gcd\limits_{H\le P}\langle\rho,\bC[G/H]\rangle_G \cr
   &=&
    \gcd\limits_{H\le P}\langle\rho|_P,\bC[P/H]\rangle_P.
\eeq
Because $\C(P)=1$ by the Ritter-Segal theorem \cite{Ritter, Segal}, we can replace the
permutation representations $\bC[P/H]$ by all rational representations of $P$
in the last term. This is clearly the same as just taking the
rational irreducible constituents ${\hat\pi}_1,...,{\hat\pi}_k$ of $\rho|_P$, so
\begin{equation}\label{eqo}
  o =
  \frac{1}{\langle\rho,\rho\rangle}\gcd\limits_j\langle \rho|_P,{\hat\pi}_j\rangle
    = \gcd\limits_j \frac{\langle\rho,\Ind_P^G{\hat\pi}_j\rangle}{\langle\rho,\rho\rangle}
    = \gcd\limits_j \mu(\rho,\Ind_P^G{\hat\pi}_j).
\end{equation}
The theorem will therefore follow from Theorem \ref{thm:innerprod}, once we show
that the gcd may be replaced by the term corresponding to any ${\hat\pi}$ of minimal
dimension. Now, by Theorem \ref{thm:innerprod} and by Lemma \ref{dimhull},
$$
\mu(\rho,\Ind_P^G{\hat\pi}_j) = \frac{\dim\hat\psi\dim{\hat\pi}_j}{\dim\rho} 
= \frac{\dim\hat\psi\,\schur(\pi_j)\dim\pi_j[\Q(\pi_j):\Q]}{\dim\rho},
$$
where $\pi_j$ is a complex irreducible constituent of ${\hat\pi}_j$.
We argue as in \cite[\S 2]{TurGcdmin}: if $p=2$, then
all the terms $\schur(\pi_j)$, $\dim\pi_j$, $[\Q(\pi_j):\Q]$ are powers
of 2, so gcd and minimum are the same. If $p$ is odd, then $\schur(\pi_j)=1$,
and moreover, either some $\pi_j=\triv$, in which case the claim is clear,
or else all $\dim\pi_j$ are powers of $p$, while all $[\Q(\pi_j):\Q]$
are $(p-1)$ times powers of $p$ (\cite[Lemma 2.1]{TurGcdmin}), so again
gcd and minimum are the same.
\end{proof}

\section{Examples: $\GL_2(\F_q)$, $\PGL_2(\F_q)$, $\SL_2(\F_q)$ and $\PSL_2(\F_q)$}
\label{sGL2}

\begin{theorem}
\label{gl2thm}
For every prime power $q=p^n$, the group $G=\GL_2(\F_q)$ has $\CH(G)~=~\{1\}$.
\end{theorem}

\begin{proof}
By Corollary \ref{mainreduction}, it suffices to show that every maximal quasi-elementary
subgroup $Q=C\rtimes P$ of $G=\GL_2(\F_q)$ is contained in some $\bar Q\<G$ with $\CH(\bar Q)=1$.
Pick $C=\langle g\rangle$ cyclic, and let $P=\Syl_l(N_G(C))$ for some prime
number $l$. Write $f(t)$ for the characteristic polynomial of $g$.

{\bf Case 1} (split Cartan). 
Suppose $f(t)$ has distinct roots $a,b\in \F_q^\times$. Then $g$ is conjugate
to $\smallmatrix a00b$, and its centraliser is the split Cartan subgroup:
$$
  Z_G(C)\iso \F_q^\times \times \F_q^\times, \qquad 
  N_G(C)\>\<\>(\F_q^\times \times \F_q^\times)\rtimes C_2,
$$
with $C_2=\langle\smallmatrix 0110\rangle$.
Here $\bar Q=(\F_q^\times \times \F_q^\times)\rtimes C_2$
has trivial $\CH(\bar Q)$ by Corollary \ref{A:Cp}.

{\bf Case 2} (non-split Cartan). 
Suppose $f(t)$ is irreducible over $\F_q$. Then the centraliser of $C$ is
the non-split Cartan subgroup:
$$
  Z_G(C)\iso \F_q[g]^\times\iso\F_{q^2}^\times, \qquad 
  N_G(C)\>\<\>\F_{q^2}^\times\rtimes\Gal(\F_{q^2}/\F_q)\iso \F_{q^2}^\times\rtimes C_2.
$$
Again $\bar Q=\F_{q^2}^\times\rtimes C_2$ has trivial $\CH$ by Corollary \ref{A:Cp}.

{\bf Case 3} (scalars). Suppose $g=\smallmatrix a00a$ is a scalar matrix. Then
$Q=C\rtimes P$ can be embedded into one of the following:
\begin{itemize}
\item if $l=p$:
$\bar Q=C\times U=C\times \{\smallmatrix 1*01\}$; in this case
$U$ is an elementary abelian $p$-group; or
\item if either $l$ is odd and $l|(q-1)$, or $l=2$ and $q\equiv 1\mod 4$:
$\bar Q=H\rtimes C_2$ with $H=$ split Cartan; or
\item if either $l$ is odd and $l|(q+1)$, or $l=2$ and $q\equiv 3\mod 4$:
$\bar Q=H\rtimes C_2$ with $H=$ non-split Cartan.
\end{itemize}
In all these cases, $\CH(\bar Q)$ is trivial by Corollary \ref{A:Cp}.

{\bf Case 4} (non-semisimple). Finally suppose that $g$ is not semisimple, 
say $g=g_s g_u$ with $g_s$ central and $g_u=\smallmatrix 1u01$ non-trivial unipotent. 
Then
\beq
  N_G C=N_G \langle g_u \rangle &=& \Bigl\{ \smallmatrix ab0c \Bigm| 
    ac^{-1}\in\F_p^\times  \Bigr\} \\[4pt]
      &=& \Bigl\{\smallmatrix a00a|a\in\F_q^\times\Bigr\}\cdot 
        \Bigl\{\smallmatrix 1c01|c\in\F_q\Bigr\} \cdot
        \Bigl\{\smallmatrix b001|b\in\F_p^\times\Bigr\}\\[5pt]
     & \cong & \F_q^\times\times (\F_q\rtimes \F_p^\times).
\eeq
If $l=p$, then $Q$ can be embedded into $\bar Q=\langle g_s\rangle \times U$,
where $U=\{\smallmatrix 1*01\}\cong \F_q$ is an elementary abelian $p$-group.
In this case $\CH(\bar Q)$ is trivial by Corollary \ref{A:Cp}.
Otherwise, $Q$ can be embedded into
$\bar Q\cong(\F_q^\times\times \langle g_u\rangle)\rtimes \F_p^\times$,
where the action in the semi-direct product is faithful. If $\tau$ is
an irreducible character of $\bar Q$ such that $\Res_{\langle g_u\rangle}\tau$
is faithful, then $\bar Q/\ker\tau$ satisfies the assumptions of
Proposition \ref{prop:mainresfaithful} with
$K=\{1\}$, so $\tr\tau\in \Perm(\bar Q)$.
Otherwise, $\Res_{\langle g_u\rangle}\tau=\dim\tau\cdot\triv$,
so $\tau$ factors through an abelian quotient, and $\tr\tau\in \Perm(\bar Q)$
e.g. by Corollary \ref{A:Cp}.
\end{proof}

\begin{remark}
It is also not hard to deduce the structure of $\CH$ for the related classical groups:

\par\noindent
\begin{itemize}
\item $G=\PGL_2(\F_q)$. Combined with Lemma \ref{lem:lifts}, the theorem implies $\CH(G)=1$. 
\item $G=\SL_2(\F_q)$.
In general, $\CH(G)\ne 1$. For example,
$\SL_2(\F_3)$ has $\C=1$ and $\CH\iso \Z/2\Z$
(it has a 2-dimensional irreducible symplectic representation), 
and $\SL_2(\F_{17})$ has $\C\iso \Z/4\Z$.
\item $G=\PSL_2(\F_q)$.
It is a result of Solomon, announced in \cite{Sol}, that $\CH(G)=1$.
This can also be seen following the argument for $\GL_2$ in Theorem \ref{gl2thm}:
the analogues of $\bar Q$ are the images of $\bar{Q}\cap \SL_2(\F_q)$ 
in $\PSL_2(\F_q)$,
and they are dihedral in Cases 1 and 2 of the theorem, elementary abelian 
or dihedral ($p=2$) in Case 3 and isomorphic to $\F_p\rtimes\F_p^\times$ in Case 4.
Again, all these groups have $\CH=1$, so $\CH(G)=1$.
\end{itemize}
\end{remark}

\section{$\PSL_n(\F_p)$}
\label{sPSL}

Let $\ord_2$ denote the 2-adic valuation of a rational number,
$\ord_2\left(2^x\cdot{\displaystyle\frac ab}\right)=x$, where $2\nmid ab$.

\begin{theorem}
\label{pslmain}
Let $k\geq 4$ be an integer, and $p$ a prime.
The groups $\PSL_k(\F_p)$, and therefore also $\SL_k(\F_p)$,
have $\CH(G)$ of exponent divisible by $2^{\min(\ord_2(k),\ord_2(p-1))}$.
\end{theorem}

In the remainder of the section we prove the theorem using Corollary
\ref{cor:muindres}.
We will construct a 2-quasi-elementary subgroup $Q=C\rtimes P$ of 
$G=\PSL_k(\F_p)$ and a rational character 
$\rho$ of $Q$ such that $\Ind_Q^G\rho$ has order divisible by
$2^{\min(\ord_2(k),\ord_2(p-1))}$ in $\CH(G)$.

\begin{lemma}\label{lem:Zsigmondy}
Let $p$ be an odd prime and $k\geq 4$ an integer. If $k=4$, assume that
$p\equiv 1\pmod 4$. Then there exists a prime number
$l$ that divides $p^{k-2}-1$ but does not divide $p^s-1$ for any $s<k-2$.
\end{lemma}
\begin{proof}
This is a special case of Zsigmondy's Theorem \cite{Zsi}.
\end{proof}

Write $Q_{2^N}$ for the generalised quaternion group of order $2^N$.

\begin{lemma}
\label{sl2syl2}
The group $\SL_2(\F_q)$, $q=p^k$ has a 2-Sylow subgroup of the form
\begin{itemize}
\item 
$S=\{\smallmatrix 1*01\}\iso C_p^k$ if $p=2$;
\item
$
  S=\langle c,h \rangle\iso Q_{2^{N}}, 
  c=\smallmatrix \alpha00{\alpha^{-1}}, 
  h=\smallmatrix 01{-1}0
$
with $\alpha\in\F_q^\times$ of exact order $2^{N-1}||q-1$, 
if $q\equiv 1\mod 4$;
\item
$
  S=\langle c,h \rangle\iso Q_{2^{N}}, 
  c=\smallmatrix \alpha{-\beta}\beta\alpha,  
  h=\smallmatrix \gamma\delta\delta{-\gamma}
$
with $\alpha+\beta \sqrt{-1}\in\F_{q^2}^\times$ of exact order $2^{N-1}||q+1$
and any choice of $\gamma,\delta\in \F_q$ with $\gamma^2+\delta^2=-1$, 
if $q\equiv 3\mod 4$.
\end{itemize}
Conjugation by the matrix $\iota=\smallmatrix{-1}001$ is an automorphism of
$S$, acting
as $-1$ in the first case, as $c\mapsto c, h\mapsto h^{-1}$ in the second case,
and as $c\mapsto c^{-1}, h\mapsto h c^{2m+1}$ for some $m$ in the last case.
\end{lemma}

\begin{proof}
Direct computation.
\end{proof}

From now on, $G$ will denote $\PSL_k(\F_p)$.
The theorem only has content when $k$ is even and $p$ is odd, so we will assume
this. Write 
$$
  n=\ord_2(p-1)\ge 1, \qquad 
  N=\ord_2(p^{k-2}-1)\ge 3, \qquad
  m=\ord_2(k-2)\ge 1.
$$

\textbf{Case A:}
Either $k>4$ or $p\equiv 1\mod 4$. 
Let $A$ be a generator of a non-split Cartan subgroup
$\F_{p^{k-2}}^\times=\GL_1(\F_{p^{k-2}})\<\GL_{k-2}(\F_p)$, and
$l$ a prime divisor of $p^{k-2}-1$ as in Lemma \ref{lem:Zsigmondy}. The conditions
on $l$ imply that the normaliser
of $\langle A^{\frac{p^{k-2}-1}l}\rangle\iso C_l$ in $\GL_{k-2}(\F_p)$ is generated by $A$ and
by the Frobenius automorphism $F\in\Gal(\F_{p^{k-2}}/\F_p)$ of order $k-2$.
Note that $F$ has determinant $-1$,
since it is an odd permutation on a normal basis of $\F_{p^{k-2}}/\F_p$.
Define 
\beq
  c_p = \mat{1&1&{} \cr {}&1&{} \cr {}&{}&{I_{k-2}}}, & & 
  c_l = \mat{1&{}&{} \cr {}&1&{} \cr {}&{}&{A^{\frac{p^{k-2}-1}l}}},\\
  x = \mat{d^{-1}&{}&{} \cr {}&1&{} \cr {}&{}&{U}}, & &
  f = \mat{-1&{}&{} \cr {}&1&{} \cr {}&{}&{F^{(k-2)/2^m}}},
\eeq
where $U=A^{\frac{p^{k-2}-1}{2^N}}$ and $d=\det U$. 
We view these matrices as representing elements of $G=\PSL_k(\F_p)$. Write 
$$
  C=\langle c_pc_l\rangle\iso C_{pl}, \;\;\;
  P=\langle x,f\rangle\iso C_{2^N}\rtimes C_{2^m}, \;\;\;
  Q=CP\iso (C_p\times C_l)\rtimes (C_{2^N}\rtimes C_{2^m}).
$$
Note that $C_{2^N}$ acts trivially on $C_l$, and through a $C_{2^n}$ quotient
on $C_p$, while $C_{2^m}$ acts through a $C_2$ quotient
on $C_p$ and faithfully on $C_l$.

\textbf{Case B:}
$p\equiv 3\mod 4$ and $k=4$. We take the same $c_p$ as in Case A, and $C=\langle c_p\rangle$.
A 2-Sylow of the centraliser of $C$ in $G$ is isomorphic to $\{1\}\times\Syl_2(\SL_2(\F_p))$,
which is isomorphic to $Q_{2^N}$ by the last case of Lemma \ref{sl2syl2}.
A 2-Sylow of the normaliser is
$$
  P=\Syl_2 N_G(C) = \Syl_2 Z_G(C) \rtimes \mat{-1&{}&{}&{}\cr{}&1&{}&{}\cr{}&{}&{-1}&{}\cr{}&{}&{}&1}
    \iso Q_{2^N}\rtimes \smallmatrix {-1}001,
$$
which is in fact isomorphic to the semi-dihedral group $SD_{2^{N+1}}$.
Again, we let $Q=CP$. 

\smallskip

In both cases, write $K$ for the centraliser of $C$ in $P$.
Thus, $K\iso C_{2^{N-n}}$ in case A,
and $K\iso Q_{2^N}$ in case B, where the isomorphism is that of Lemma
\ref{sl2syl2}. Let $A_p$ be $K$ in case A, and a cyclic subgroup of index
2 in $K$ that is normal in $Q$ in case B.
Let $\chi$ be faithful irreducible characters of $CA_p$, 
$\tau=\Ind_{CA_p}^Q\chi$ and $\rho=\tr\tau\in \Irr_{\Q}(Q)$.

\begin{lemma}\label{ordrho}
The character $\rho$ has order $2^n$ in $\CH(Q)$.
\end{lemma}
\begin{proof}
We will use Proposition \ref{prop:mainresfaithful}.
The biggest subgroup of $P$ that intersects $CA_p$ trivially is of order
2 in case B, and of order $2^m$ in case A.
So the order of $\rho$ in $\CH(G)$ is $2^{N+1-(N-1)-1}=2$ in case B, and
$2^{N+m-(N-n)-m}=2^n$ in case A.
\end{proof}

Finally, we show that
$\Ind_Q^G\rho$ has order divisible by $2^{\min(\ord_2(k),\ord_2(p-1))}$ in $\CH(G)$.
We will use Corollary \ref{cor:muindres} with $Q_1=Q_2=Q$ and $\chi_1=\chi_2=\chi$.
In view of Lemma \ref{ordrho}, it suffices
to show that
\begin{eqnarray}\label{muCaseA}
\sum_{[U]\in CC_{\cyc}(G)} S(U)
\end{eqnarray}
has 2-adic valuation at most $n-\min(\ord_2(k),\ord_2(p-1))$, where
for $U\leq CA_p$,
$$
S(U) = \frac{[\Q(\tau):\Q]}{|CA_p|^2}|N_G(U)|\phi(|U|)
  \cdot \biggl(\sum_{D\leq CA_p\atop D\sim U}
  \frac{\mu([D:D\cap\ker\chi])}{\phi([D:D\cap\ker\chi])}\biggr)^2.
$$
Note that since $CA_p$ is cyclic and $\chi$ is faithful, this simplifies to
$$
S(U) = \frac{[\Q(\tau):\Q]}{|CA_p|^2\phi(|U|)}|N_G(U)|\mu(|U|)^2,
$$
see Remark \ref{simplification}.
In particular, $S(U)=0$ if $U$ has non-square-free order.

\textbf{Case A}.

The subgroups of $CK$ of square-free
order are $C_{2lp}$, $C_{lp}$, $C_{2l}$, $C_l$, $C_{2p}$, $C_{p}$,
$C_2$, and $C_1$. We will show that $S(C_{lp}) + S(C_{2lp})$ has a strictly
lower 2-adic valuation than the rest of the sum, and that this valuation
is $n-\min(\ord_2(k),n)$. A summary of the calculations that follow is:

\begin{eqnarray*}
\ord_2 [\Q(\tau):\Q] &=& \ord_2(l-1)+N-n-1-m,\\
\ord_2 |CK|^2 &=& 2(N-n),\\
\phi(|C_{lp}|) = \phi(|C_{2lp}|) &=& (l-1)(p-1),\\
|N_G(C_{lp})|=|N_G(C_{2lp})| &=& \frac{(k-2)p(p^{k-2}-1)(p-1)}{\gcd (k,p-1)},\\
\ord_2(S(C_{lp})+S(C_{2lp})) &=& \ord_2(2S(C_{lp}))\\
& = &1+\ord_2(l-1)+N-n-1-m-2(N-2) + N+\\
& & n + m -\min(\ord_2(k),n) - \ord_2(l-1)+n\\
& = & n-\min(\ord_2(k),n).
\end{eqnarray*}
The assertions concerning $|CK|$ and $\phi(|C_{lp}|)$ are clear.

Since the conjugation action of $P$ on $\Irr(CK)$ is through Galois automorphisms,
and $\ker(P\rightarrow \Aut(CK))$ has index $2^{n+m}$ in $P$, we have
$$
[\Q(\tau):\Q]=2^{-n-m}[\Q(\chi):\Q]=\frac{p-1}{2^n}\frac{(l-1)2^{N-n-1}}{2^m},
$$
with 2-adic valuation $\ord_2(l-1)+N-n-1-m$.


The normaliser $N_{\GL_k(p)}$ of the preimage of $C_{lp}$ under $\SL\rightarrow \PSL$
consists of block diagonal matrices, with the normaliser of non-split Cartan
in the lower right corner (order $(k-2)(p^{k-2}-1)$), and a Borel subgroup in the top left
(order $p(p-1)^2$). The determinant is surjective on $N_{\GL_k(p)}$,
and $N_{\GL_k(p)}$ contains $Z(\GL_k(p))$, so the normaliser of $C_{lp}$ in $\PSL$ has order
$\frac{(k-2)(p^{k-2}-1)p(p-1)}{\gcd(k,p-1)}$, with 2-adic
valuation $N+n+m-\min(\ord_2(k),n)$. This is also the normaliser of $C_{2lp}$.

It remains to show that the rest of the sum in equation (\ref{muCaseA})
has strictly greater 2-adic valuation than $\ord_2(S(C_{lp})+S(C_{2lp}))$.
If $U\leq C$, then $|N_G(U)|$ and  $|N_G(UC_2)|$ agree up to a power of $p$,
$\phi(|U|) = \phi(|UC_2|)$,
while $\mu(|U|)=-\mu(|UC_2|)$. It follows that the 2-adic valuation of $S(U) +
S(UC_2)$ is at least 1 greater than that of $S(U)$.

Moreover, for any $U\leq C_{lp}$, the normaliser of $U$ in $G$ contains
that of $C_{lp}$, while $1/\phi(|U|)$ has strictly greater 2-adic valuation than 
$1/\phi(|C_{lp}|)$ whenever $U\neq C_{lp}$. This establishes the claim.
\newline

\textbf{Case B}.
The subgroups of $CA_p$ of square-free order are $C_1$, $C_2$, $C_p$, and $C_{2p}$.
We will show that $\ord_2(\sum S(U)) = \ord_2(S(C_p)+S(C_{2p}))=0$.
Again, we summarise the calculations as follows:
\begin{eqnarray*}
\ord_2 [\Q(\tau):\Q] &=& N-3,\\
\ord_2 |CA_p|^2 &=& 2N-2,\\
\phi(|C_{p}|) = \phi(|C_{2p}|) &=& p-1,\\
|N_G(C_{p})|=p^4|N_G(C_{2p})| &=& p^4\cdot\frac{(p-1)^3p^2(p+1)}{2},\\
\ord_2(S(C_{p})+S(C_{2p})) &=& \ord_2((1+p^4)S(C_{2p}))\\
& = &1+N-3-2N+2-1 + N+1=0.
\end{eqnarray*}
The assertions concerning $|CA_p|$ and $\phi$ are clear.

It follows from the description of the $P$-action on $\Irr(CK)$
that $[\Q(\tau):\Q]=\frac12[\Q(\chi):\Q]$, and has 2-adic valuation $2^{N-3}$.

The normaliser of $C_{2p}$ in $\GL_4$ is block diagonal, with all invertible matrices
in the bottom right corner, and Borel in the top left. So its order in $\PSL$ is
$\frac{(p-1)^3p^2(p+1)}{2}$ with 2-adic valuation $N+1$. 
Finally, $|N(C_p)|=p^4|N(C_{2p})|$, e.g. see Murray \cite{Mur} \S4.

It remains to show that the 2-adic valuation of $S(C_1)+S(C_2)$ is positive.
The normaliser of $C_2$ in $\GL_4$ is $\GL_2\times \GL_2$, so the order of the normaliser
in $\PSL$ is $\frac{(p-1)^3p^2(p+1)^2}{2}$, with 2-adic valuation $2N$, and the
normaliser of $C_1$ is even bigger. So the 2-adic valuations of $S(C_1)$ and of $S(C_2)$
are positive.

\begin{corollary}\label{cor:psl}
As $G$ ranges over the simple groups $\PSL_{k}(\F_p)$, and therefore also
over $\SL_k(\F_p)$, the exponent of $\C(G)_2$ is unbounded.
\end{corollary}
\begin{proof}
If $\ord_2(k) > \ord_2(p-1)$, then by
\cite[Lemma 5.6(1)]{TurSchur} all Schur indices in $\PSL_k(\F_p)$ are trivial.
So the assertion follows from Theorem \ref{pslmain}.
\end{proof}


\end{document}